\documentclass[11pt, reqno]{amsart}
\usepackage{amssymb}
\usepackage{amsmath}
\allowdisplaybreaks[1]
\usepackage{graphicx,color}
\usepackage{wrapfig,framed,bbm}
\usepackage{mathtools}
\usepackage[height=22cm, width=16.5cm, hmarginratio={1:1}]{geometry}
\usepackage{hyperref}

\newcommand{\p}{\partial}
\newcommand{\nn}{\nonumber}
\newcommand{\e}{\epsilon}

\newcommand{\bt}{{\bf t}}

\newcommand{\F}{\mathcal{F}}
\newcommand{\CC}{\mathbb{C}}

\newcommand{\bT}{{\bf T}}

\newcommand{\beq}{\begin{equation}}
\newcommand{\eeq}{\end{equation}}

\numberwithin{equation}{section}
\theoremstyle{plain}
\newtheorem{theorem}{Theorem}
\newtheorem{proposition}{Proposition}

\newtheorem{corollary}{Corollary}
\theoremstyle{definition}

\newtheorem{remark}{Remark}
\newtheorem{example}{Example}

\def\1{\mathbbm{1}}

\begin{document}
\title[Degree zero GW invariants of target curves]{Degree zero Gromov--Witten invariants for smooth curves}
\author{Di Yang}
\address{School of Mathematical Sciences, University of Science and Technology of China, Hefei 230026, P.R. China}
\email{diyang@ustc.edu.cn}
\date{}
\maketitle

\begin{abstract}
For a smooth projective curve, 
we derive a closed formula for the generating series of its Gromov--Witten invariants in 
genus $g$ and degree zero.
It is known that the calculation of these invariants can be reduced to that of  
the $\lambda_g$ and $\lambda_{g-1}$ 
integrals on the moduli space of stable algebraic curves. 
The closed formula for the $\lambda_g$ integrals is given by the $\lambda_g$ conjecture, proved by 
Faber and Pandharipande. 
We compute in this paper the $\lambda_{g-1}$ integrals via solving the degree 
zero limit of the loop equation associated to the complex projective line. 
\end{abstract}

\section{Introduction}\label{section1}
Let $X$ be a smooth projective variety of complex dimension~$D$, and let $X_{g,n,\beta}$ be the moduli space of 
stable maps of degree $\beta\in H_2(X,\mathbb{Z})/{\rm torsion}$ 
with target~$X$ from curves of genus~$g$ with~$n$ distinct marked points. 
Here, $g,n\ge0$. 
Choose a homogeneous basis $\phi_1 = 1,\phi_2,\dots,\phi_l$ of 
the cohomology ring $H^{*}(X;\mathbb{C})$ 
with $\phi_\alpha \in H^{2q_\alpha}(X;\CC)$, $\alpha=1,\dots,l$, 
where $0=q_1<q_2 \leq q_3 \leq \dots \leq q_{l-1} < q_l = D$. Note that $q_\alpha$ is a half integer if $\phi_\alpha$ is an odd degree class. 
We denote the Poincar\'e pairing on 
$H^{*}(X;\mathbb{C})$ by $\langle\,,\,\rangle$.
The integrals
\beq
\int_{[X_{g,n,\beta}]^{\rm virt}} c_1(\mathcal{L}_1)^{i_1} {\rm ev}_1^*(\phi_{\alpha_1}) \cdots c_1(\mathcal{L}_n)^{i_n} {\rm ev}_n^*(\phi_{\alpha_n}), \quad 
\alpha_1,\dots,\alpha_n=1,\dots,l, \, i_1,\dots,i_n\geq0,
\eeq 
are called {\it Gromov--Witten (GW) invariants of~$X$ of genus~$g$ and degree~$\beta$}. 
Here, ${\rm ev}_a$, $a=1,\dots,n$, is the evaluation map, 
$\mathcal{L}_a$ is the $a$th tautological line bundle on $X_{g,n,\beta}$, 
and $[X_{g,n,\beta}]^{\rm virt}$ is the {\it virtual fundamental class}, which is an element in the Chow ring having the complex dimension 
\beq\label{vdim}
(1-g)(D-3) + n + \langle \beta, c_1(X)\rangle. 
\eeq 
The {\it genus~$g$ free energy} for the GW invariants of~$X$ is defined as the generating series
\beq\label{defF}
\F^X_g({\bf t};q) = \sum_{n,\beta}\sum_{\alpha_1,\dots,\alpha_n \atop i_1,\dots,i_n} \frac{q^\beta t^{\alpha_1,i_1} \cdots t^{\alpha_n,i_n} }{n!} 
\int_{[X_{g,n,\beta}]^{\rm virt}} c_1(\mathcal{L}_1)^{i_1} {\rm ev}_1^*(\phi_{\alpha_1}) \cdots c_1(\mathcal{L}_n)^{i_n} {\rm ev}_n^*(\phi_{\alpha_n}),
\eeq
where ${\bf t}=(t^{\alpha,i})_{\alpha=1,\dots,l,\,i\ge0}$ is an infinite vector of indeterminates, 
and 
\beq
q^\beta=q_1^{m_1} \cdots q_r^{m_r} ~ ({\rm for}~\beta=m_1\beta_1+\dots+m_r\beta_r)
\eeq
is an element of the Novikov ring. Here, $(\beta_1,\dots,\beta_r)$ is a basis of $H_2(X;\mathbb{Z})/{\rm torsion}$. 
The sum 
$$
\sum_{g\ge0} \e^{2g-2} \F^X_g({\bf t};q) =: \F^X({\bf t}; \e, q)
$$
is called the {\it free energy} for the GW invariants of~$X$.
It should be noted that the order of each monomial $t^{\alpha_1,i_1} \cdots t^{\alpha_n,i_n}$ in~\eqref{defF} is important since the odd cohomology classes are considered and that 
in some literature, the monomial $t^{\alpha_1,i_1} \cdots t^{\alpha_n,i_n} $ in~\eqref{defF} is ordered as $t^{\alpha_n,i_n} \cdots t^{\alpha_1,i_1}$.

The interest of this paper is on the calculation of degree zero GW invariants. 
In particular, we will focus on the case 
when the target is a smooth projective curve. 
Before specializing to curves, we will first review what is known about the 
degree zero invariants for a general smooth projective variety $X$.

The moduli space $X_{g,n,0}$ is isomorphic to $\overline{\mathcal{M}}_{g,n} \times X$, 
where $\overline{\mathcal{M}}_{g,n}$ denotes the 
Deligne--Mumford moduli space~\cite{DM69} of stable algebraic curves of 
genus~$g$ with $n$ distinct marked points. With this identification, 
\beq
[X_{g,n,0}]^{\rm virt} = e(E_{g,n}^\vee \boxtimes T_X) \cap [\overline{\mathcal{M}}_{g,n} \times X],
\eeq
where  
$\mathbb{E}_{g,n}$ is the Hodge bundle on $\overline{\mathcal{M}}_{g,n}$, $T_X$ is the tangent bundle of~$X$, 
and $e(\mathbb{E}_{g,n}^{\vee} \boxtimes T_X)$ is the Euler class of the obstruction bundle on $\overline{\mathcal{M}}_{g,n} \times X$. 
Therefore, we have the following formula for the degree zero GW invariants of~$X$:
\beq\label{kme}
\int_{[X_{g,n,0}]^{\rm virt}} c_1(\mathcal{L}_1)^{i_1} {\rm ev}_1^*(\phi_{\alpha_1}) \cdots c_1(\mathcal{L}_n)^{i_n} {\rm ev}_n^*(\phi_{\alpha_n}) = \int_{\overline{\mathcal{M}}_{g,n} \times X} 
 \psi_1^{i_1} \cdots \psi_n^{i_n}  \phi_{\alpha_1}\cdots\phi_{\alpha_n} \cup e(\mathbb{E}_{g,n}^{\vee} \boxtimes T_X).
\eeq
Here $\psi_a$, $a=1,\dots,n$, denotes the first Chern class of the $a$th tautological line bundle on 
$\overline{\mathcal{M}}_{g,n}$. We see from formula~\eqref{kme} that Hodge classes 
on~$\overline{\mathcal{M}}_{g,n}$ enter into the story, and we use 
 $\lambda_j$ ($j=0,\dots,g$) denote the $j$th Chern class of $\mathbb{E}_{g,n}$.
For more details and references about formula~\eqref{kme} see e.g.~\cite{KM}. 

Denote by $\F^X_{g,\,\deg=0}({\bf t})$ the degree zero part of the genus~$g$ free energy $\F^X_g({\bf t};q)$, i.e.,
\beq\label{defF0}
\F^X_{g,\,\deg=0}({\bf t}) = \sum_{n} \sum_{\alpha_1,\dots,\alpha_n \atop i_1,\dots,i_n} \frac{t^{\alpha_1,i_1} \cdots t^{\alpha_n,i_n}}{n!}  
\int_{[X_{g,n,0}]^{\rm virt}} c_1(\mathcal{L}_1)^{i_1} {\rm ev}_1^*(\phi_{\alpha_1}) \cdots c_1(\mathcal{L}_n)^{i_n} {\rm ev}_n^*(\phi_{\alpha_n}).
\eeq
Following~\cite{Du14, GP, KM}, let us apply formula~\eqref{kme} for the computation of $\F^X_{g,\,\deg=0}({\bf t})$. 

In {\bf \textit{genus zero}}, 
\begin{align}
&\int_{[X_{0,n,0}]^{\rm virt}} c_1(\mathcal{L}_1)^{i_1} {\rm ev}_1^*(\phi_{\alpha_1}) \cdots c_1(\mathcal{L}_n)^{i_n} {\rm ev}_n^*(\phi_{\alpha_n}) 
 = \int_{\overline{\mathcal{M}}_{0,n}} \psi_1^{i_1} \cdots \psi_n^{i_n} \int_X  \phi_{\alpha_1}\cdots\phi_{\alpha_n}.
\end{align}
Together with the well-known formula $\int_{\overline{\mathcal{M}}_{0,n}} \psi_1^{i_1} \cdots \psi_n^{i_n}=\binom{n-3}{i_1,\dots,i_n}$, 
one obtains
\beq\label{eq111-710}
\F^X_{g=0,\,\deg=0}({\bf t}) = 
\sum_{n\geq3} \frac1{n(n-1)(n-2)} \sum_{i_1+\cdots+i_n=n-3} \int_X \frac{{\bf t}^{i_1}\cdots{\bf t}^{i_n}}{i_1! \cdots i_n!},
\eeq
where ${\bf t}^i:=\sum_{\alpha} t^{\alpha,i}\phi_\alpha$, $i\geq0$, are the cohomology-valued times.

In {\bf \textit{genus one}}, we know from~\cite{KM} that 
\beq\label{eulerg1}
e(\mathbb{E}_{1,n}^{\vee} \boxtimes T_X) = c_D(\mathbb{E}_{1,n}^{\vee} \boxtimes T_X) =  c_D(X) - \lambda_1 c_{D-1}(X), \quad n\ge1,
\eeq
and so,
\begin{align}
&\int_{[X_{1,n,0}]^{\rm virt}} c_1(\mathcal{L}_1)^{i_1} {\rm ev}_1^*(\phi_{\alpha_1}) \cdots c_1(\mathcal{L}_n)^{i_n} {\rm ev}_n^*(\phi_{\alpha_n}) \nn\\
&\quad = \int_{\overline{\mathcal{M}}_{1,n}} \psi_1^{i_1} \cdots \psi_n^{i_n} \int_X  \phi_{\alpha_1}\cdots\phi_{\alpha_n} c_D(X) - 
\int_{\overline{\mathcal{M}}_{1,n}} \psi_1^{i_1} \cdots \psi_n^{i_n} \lambda_1  \int_X  \phi_{\alpha_1}\cdots\phi_{\alpha_n} c_{D-1}(X).
\end{align}
Combined with explicit computations given in e.g.~\cite{Du14, DLYZ, DY}, it follows that 
\beq
\F^X_{g=1, \,\deg=0}({\bf t}) = \frac1{24} \bigl\langle c_D(X), \log V_1({\bf t}) \bigr\rangle 
- \frac1{24} \bigl\langle c_{D-1}(X), V({\bf t}) \bigr\rangle.
\eeq
Here and below,  for ${\bf T}=(T_0,T_1,T_2,\dots)$,
\beq\label{VT}
V(\bT) := \sum_{n\ge1} \frac1n \sum_{i_1+\dots+i_n=n-1} \frac{T_{i_1} \dots T_{i_n}}{i_1! \dots i_n!},  \quad V_k(\bT):=\frac{\p^k V(\bT)}{\p T_0^k} ~( k\geq0),
\eeq
  and 
 $V_k({\bf t})$, $k\ge0$, are understood as replacing $T_i$ by ${\bf t}^i$, $i\geq0$.

In {\bf \textit{genus bigger than or equal to~two}}, we divide the consideration into several cases. 
For the case when the dimension $D$ is bigger than~3, the virtual dimension~\eqref{vdim} is negative and 
$\F^X_{g,\deg=0}({\bf t})$ vanishes. 
For the case when $D=3$, i.e., $X$ is a {\it threefold}, Getzler and Pandharipande~\cite{GP} obtain 
\beq
e(\mathbb{E}_{g,n}^{\vee} \boxtimes T_X) 
= (-1)^g \bigl(c_3(X)-c_2(X)c_1(X)\bigr) \lambda_g\lambda_{g-1}\lambda_{g-2}, \quad g\ge2.
\eeq
Using this formula and the well-known formula (cf.~\cite{Fa, FP0})
\beq
\int_{\overline{\mathcal{M}}_{g,0}} \lambda_g\lambda_{g-1}\lambda_{g-2} = 
\frac{1}{2(2g-2)!} \frac{|B_{2g-2}|}{2g-2} \frac{|B_{2g}|}{2g}, \quad g\geq2, 
\eeq
the following closed expression for $\F^X_{g, \,\deg=0}({\bf t})$ 
is obtained in~\cite{DLYZ}:
\beq
\F^X_{g, \,\deg=0}({\bf t}) = \frac{(-1)^g}{2(2g-2)!} \frac{|B_{2g-2}|}{2g-2} \frac{|B_{2g}|}{2g} 
\bigl\langle c_3(X)-c_2(X)c_1(X), V_1({\bf t})^{2g-2}\bigr\rangle,
\quad g\geq2.
\eeq
Here and below, $B_k$ denote the Bernoulli numbers. When $D=0$, i.e., $X$ is a {\it point}, the GW invariants of~$X$ can be uniquely determined by the 
celebrated Witten-Kontsevich theorem (cf.~\cite{Konts, Witten}). 
The remaining cases are the $D=1$ case and the $D=2$ case, i.e., $X$ is a {\it curve} or a {\it surface}.
The focus of this paper is the curve case, and we leave the surface case yet to a future publication.

From now on we assume that $X$ is a smooth projective curve of genus~$h$. 
Getzler and Pandharipande~\cite{GP} show that  
\beq
e(\mathbb{E}_{g,n}^{\vee} \boxtimes T_X) = (-1)^g \bigl(\lambda_g - \lambda_{g-1}c_1(X)\bigr), \quad g\geq2
\eeq
(also true for $g=1$ being compared with~\eqref{eulerg1}).
So for all $g\geq1$, we have
\begin{align}
&(-1)^g\int_{[X_{g,n,0}]^{\rm virt}} 
c_1(\mathcal{L}_1)^{i_1} {\rm ev}_1^*(\phi_{\alpha_1}) \cdots c_1(\mathcal{L}_n)^{i_n} {\rm ev}_n^*(\phi_{\alpha_n}) \nn\\
& ~ = \int_{\overline{\mathcal{M}}_{g,n}} \psi_1^{i_1} \cdots \psi_n^{i_n} \lambda_g \int_X  \phi_{\alpha_1}\cdots\phi_{\alpha_n} - 
\int_{\overline{\mathcal{M}}_{g,n}} \psi_1^{i_1} \cdots \psi_n^{i_n} \lambda_{g-1}  \int_X  \phi_{\alpha_1}\cdots\phi_{\alpha_n} c_1(X). \label{curveobs}
\end{align}
In terms of generating series, we have
\begin{align}
(-1)^g\F^X_{g, \,\deg=0}({\bf t})
= \bigl\langle 1, \mathcal{H}_g(\lambda_g;{\bf t})\bigr\rangle -  \bigl\langle c_1(X), \mathcal{H}_g(\lambda_{g-1}; {\bf t})\bigr\rangle, \quad 
g\geq1. \label{genfxh}
\end{align}
Here, for an element $\varphi\in H^*(\overline{\mathcal{M}}_{g,n};\CC)$, and for $\bT=(T_0,T_1,T_2,\dots)$,
\beq\label{genHodgegamma}
\mathcal{H}_g(\varphi; \bT) := 
\sum_n \sum_{i_1,\dots,i_n}\int_{\overline{\mathcal{M}}_{g,n}} \psi_1^{i_1} \cdots \psi_n^{i_n} \varphi \frac{T_{i_1} \dots T_{i_n}}{n!},
\eeq
and $\mathcal{H}_g(\varphi;{\bf t})$, $g\geq1$, are understood by replacing $T_i$ by ${\bf t}^i$, $i\geq0$.
We note that the power series $V(\bT)$ 
given in~\eqref{VT} is just $\p^2 \mathcal{H}_0(1;\bT)/\p T_0^2$.
It remains to compute $\mathcal{H}_g(\lambda_g;\bT)$ and $\mathcal{H}_g(\lambda_{g-1};\bT)$; in other words, we need to compute 
\beq\label{gg1}
\int_{\overline{\mathcal{M}}_{g,n}} \psi_1^{i_1} \cdots \psi_n^{i_n} \lambda_g, 
\qquad \int_{\overline{\mathcal{M}}_{g,n}} \psi_1^{i_1} \cdots \psi_n^{i_n} \lambda_{g-1} .
\eeq

It follows from the $\lambda_g$ conjecture~\cite{GP} that the Hodge integrals in~\eqref{gg1} with~$\lambda_g$ have the explicit 
 expression:
\beq\label{lambdagformu}
\int_{\overline{\mathcal{M}}_{g,n}} \psi_1^{i_1} \cdots \psi_n^{i_n} \lambda_g = 
\binom{2g+n-3}{i_1,\dots,i_n} \frac{2^{2g-1}-1}{2^{2g-1}} \frac{|B_{2g}|}{(2g)!},\quad g\geq1.
\eeq
The $\lambda_g$ conjecture is proved in~\cite{FP1}. In~\cite{DY} it is shown that 
formula~\eqref{lambdagformu} yields 
\beq\label{123}
\mathcal{H}_g (\lambda_g; \bT) = 
\frac{2^{2g-1}-1}{2^{2g-1}} \frac{|B_{2g}|}{(2g)!} V_{2g-2}(\bT), \quad g\geq1.
\eeq 

Computation for the Hodge integrals in~\eqref{gg1} with $\lambda_{g-1}$ is more involved. 
In~\cite{Du14, DLYZ},  
closed expressions for $\mathcal{H}_g (\lambda_{g-1};\bT)$ for the first few values of~$g$ are found; for example,
\begin{align}
&\mathcal{H}_1 (\lambda_0=1;\bT) = \frac1{24} \log V_1, \quad \mathcal{H}_2 (\lambda_1;\bT) = 
\frac1{480}\frac{V_3}{V_1} - \frac{11}{5760} \frac{V_2^2}{V_1^2},\\
& \mathcal{H}_3 (\lambda_2;\bT) = -\frac{19V_2^4}{53760V_1^4} + \frac{151}{207360}\frac{V_2^2V_3}{V_1^3} - \frac{61V_3^2}{322560V_1^2}
-\frac{373V_2V_4}{1451520V_1^2} + \frac{41V_5}{580608V_1}. 
\end{align}
Note that on the right-hand sides, we omitted the arguments~$\bT$ in~$V_m(\bT)$. 

Before proceeding, we recall some notations. A {\it partition}  
is a non-increasing infinite sequence of non-negative integers $\mu=(\mu_1,\mu_2,\dots)$. 
The number of non-zero components of~$\mu$ is called the {\it length} of~$\mu$, denoted by~$\ell(\mu)$. 
The sum $\sum_{i\ge1} \mu_i$ is called the {\it weight} of~$\mu$, denote by~$|\mu|$. 
The set of all partitions is denoted by $\mathcal{P}$, and the set of partitions of weight~$k$ is denoted by $\mathcal{P}_k$. 
A partition~$\mu$ of weight~$k$ is also called a {\it partition of~$k$}. If the length of the partition $\mu$ is positive, 
it is often denoted by $\mu=(\mu_1,\dots,\mu_{\ell(\mu)})$; otherwise, it can be denoted either as $(0)$ or as $(\,)$. 
Denote $\mu+1=(\mu_1+1,\dots,\mu_{\ell(\mu)}+1)$ if $\ell(\mu)>0$, and $(\,)+1=(\,)$ otherwise. 
The expression $m_i(\mu)$ will denote the {\it multiplicity} of~$i$ in~$\mu$, $i\geq1$, and denote 
$m(\mu)!=\prod_{i=1}^{\infty} m_i(\mu)!$. 
For a sequence of indeterminates $(y_0, y_1, y_2,\dots)$, 
 $y_\mu:=\prod_{i=1}^{\ell(\mu)} y_{\mu_i}$ (clearly, $y_{(\,)}=1$).

According to \cite{DLYZ, DY}, for $g\geq1$ there exist
functions $W_g(V_1,\dots,V_{2g-1})$ of $(2g-1)$ variables of the form 
\beq\label{126713}
W_1(V_1) = \frac{\log V_1}{24}, \quad W_g(V_1,\dots, V_{2g-1}) 
=  \sum_{\mu\in \mathcal{P}_{2g-2}} c^g_\mu \frac{V_{\mu+1}}{V_1^{\ell(\mu)}}~ (g\geq2), 
\eeq
such that
\beq
\mathcal{H}_g (\lambda_{g-1};\bT) = W_g(V_1(\bT),\dots,V_{2g-1}(\bT)), \quad g\geq1.
\eeq
Here 
$c^g_\mu$ are constants. For example, $W_2(V_1,V_2,V_3)=\frac1{480}\frac{V_3}{V_1}-\frac{11}{5760} \frac{V_2^2}{V_1^2}$.
It then follows from formula~\eqref{genfxh} the following proposition. 
\begin{proposition}\label{prop1}
For $g\geq1$, the genus~$g$ free energy of~$X$ of degree zero has the expression:
\beq\label{Fprop1}
\F^X_{g, \,\deg=0}({\bf t}) = 
(-1)^g \frac{2^{2g-1}-1}{2^{2g-1}} \frac{|B_{2g}|}{(2g)!} \langle 1, V_{2g-2}(\bt)\rangle - (-1)^g (2-2h) W_g(V_1({\bf P}),\dots, V_{2g-1}({\bf P})).
\eeq 
Here ${\bf P}:=(t^{1,0},t^{1,1},\dots)$.
\end{proposition}

The goal of this paper is to give a closed formula for $W_g(V_1,\dots,V_{2g-1})$.

Introduce some more notations. 
 The Lagrange number $L(\mu)$ is defined by 
\beq
L(\mu) = \frac{(|\mu|+\ell(\mu))! (-1)^{\ell(\mu)}}{m(\mu)! \prod_{j\ge1} (j+1)!^{m_j(\mu)}}.
\eeq
For $g\geq1$, introduce the rational function $\mathcal{B}_g(\lambda; {\bf V})$ as follows:
\begin{align}
& \mathcal{B}_g(\lambda; {\bf V}) \nn\\
&= \frac{2^{2g-1}-1}{2^{2g-1}} \frac{|B_{2g}|}{(2g)!} \biggl(\frac1{(\lambda-V)^2}\biggr)_{2g-2} 
+ \frac{2^{2g-1}-1}{2^{2g-1}} \frac{|B_{2g}|}{(2g)!} \sum_{k=1}^{2g-2} \binom{2g-2}{k} 
\biggl(\frac1{\lambda-V}\biggr)_{k-1}\biggl(\frac1{\lambda-V}\biggr)_{2g-1-k} \nn\\
& \quad - \frac12\sum_{g_1+g_2=g \atop g_1,g_2\geq1} \frac{2^{2g_1-1}-1}{2^{2g_1-1}} \frac{|B_{2g_1}|}{(2g_1)!} \frac{2^{2g_2-1}-1}{2^{2g_2-1}} 
\frac{|B_{2g_2}|}{(2g_2)!} \biggl(\frac1{\lambda-V}\biggr)_{2g_1-1}\biggl(\frac1{\lambda-V}\biggr)_{2g_2-1}, \label{defBg710}
\end{align}
where ${\bf V}=(V_0=V,V_1,V_2,\dots)$, and for a function $f({\bf V})$ and $r\ge0$, $(f({\bf V}))_r$ means $\p^r (f({\bf V}))$ 
with $\p:=\sum_k V_{k+1} \p/\p V_k$.
Denote $\mathcal{B}_{g,j}={\rm Coef}\bigl((\lambda-V)^{-j-1},\mathcal{B}_g(\lambda;{\bf V})\bigr)$. 
We will prove in Section~\ref{section3} the following theorem. 
\begin{theorem}\label{thm1}
For $g\geq2$, we have
\beq\label{WghomocB}
W_g(V_1,\dots,V_{2g-1}) 
= \frac1{2g-2} \sum_{k=2}^{2g-1} (k-1) V_k \sum_{j=1}^{2g-1} c_{k,j} \mathcal{B}_{g,j}
\eeq
with 
\beq
c_{i,j} := \frac{1}{j!} \sum_{\mu\in \mathcal{P}_{j-i}} \binom{\ell(\mu)+j-1}{i-1} L(\mu)  \frac{V_{\mu+1}}{V_1^{\ell(\mu)+j}}.
\eeq
\end{theorem}
Let us briefly describe the idea of the proof. 
Observe from~\eqref{Fprop1} that the functions $W_g$ in~\eqref{Fprop1} are independent of~$X$, 
so we can compute $W_g$ by taking $X=\mathbb{P}^1$. The partition function
$$Z^{\mathbb{P}^1}({\bf t};\e,q) := \exp \bigl(\F^{\mathbb{P}^1}({\bf t};\e,q)\bigr)$$
 is deeply connected with integrable systems \cite{DZ, OP2, OP1}, so one can apply the theory 
 of integrable systems (Lax pairs, Virasoro constraints, etc.)
 to get closed expressions of them; for example, explicit $n$-point functions are 
 obtained in~\cite{DYZ-gwp1, GOP-GWp1} by using the Toda lattices. In this paper,  
 we will employ the loop equation~\cite{DZ} for~$\mathbb{P}^1$.

As a special case of Theorem~\ref{thm1}, we will give in Section~\ref{section4} a simple proof of the following theorem. 

\noindent {\bf Theorem A} (\cite{FP0, GP}). {\it For $g\geq1$, the following formula holds:
\begin{align}
&\int_{\overline{\mathcal{M}}_{g,1}} \psi_1^{2g-1} \lambda_{g-1} \nn\\
&= \frac{2^{2g-1}-1}{2^{2g-1}} \frac{|B_{2g}|}{(2g)!}\sum_{k=1}^{2g-1}\frac1k 
- \frac1{2^{2g-1}(2g-1)!}\sum_{g_1+g_2=g\atop g_1, g_2>0} \bigl(2^{2g_1-1}-1\bigr)\bigl(2^{2g_2-1}-1\bigr) 
\frac{|B_{2g_1}|}{2g_1}\frac{|B_{2g_2}|}{2g_2}. \label{corgm1}
\end{align}}

\smallskip

The paper is organized as follows. In Section~\ref{section2} we review the loop equation for 
GW invariants of~$\mathbb{P}^1$. In Section~\ref{section3} we prove Theorem~\ref{thm1}. 
Some applications are discussed in Section~\ref{section4}. 

\section{Loop equation for GW invariants of~$\mathbb{P}^1$}\label{section2}
In this section, we review the loop equation for GW invariants of the complex projective line, 
that can be derived using the structure of the associated Frobenius manifold 
as a result of the following three properties:
\begin{itemize}
\item [1)] Virasoro constraints for the genus zero free energy;

\item [2)] Existence of jet-variable representation for the higher genera free energies;

\item [3)] Virasoro constraints for the all-genus free energy. 
\end{itemize}

Recall that the Frobenius manifold~\cite{Du96} associated to the GW invariants of~$\mathbb{P}^1$ has the  
potential 
\beq
F = \frac12 (v^1)^2 v^2 + q e^{v^2}.
\eeq
Here $(v^1,v^2)$ is a system of flat coordinates for the invariant flat metric of this Frobenius manifold with 
$\p_{v^1}$ being the unity vector field, and 
we will also use the notation $v=v^1, u=v^2$. 
The {\it principal hierarchy} associated to this Frobenius manifold is a hierarchy of commuting evolutionary
PDEs given by
\begin{align}\label{phdefp1}
\frac{\p v^\alpha}{\p t^{\beta,i}} = 
\sum_{\rho} \eta^{\alpha\rho}\p_x \biggl(\frac{\p \theta_{\beta,i+1}(v,u;q)}{\p v^\rho}\biggr), \quad i\geq0, \, \alpha,\beta=1,2,
\end{align}
where $\eta^{\alpha\rho}=\delta_{\alpha+\rho,3}$, and 
$(\theta_{\alpha,k}(v^1,v^2;q))_{\alpha=1,2, k\geq0}$ 
are holomorphic functions which can be 
 defined via the generating series~\cite{DZ-norm, DZ}
\begin{align}
&\theta_1(v,u;q;z):=\sum_{k\geq0} \theta_{1,k}(v,u;q) z^k
= -2e^{zv}\sum_{m\geq0}\Bigl(\gamma-\frac12 u+\psi(m+1)\Bigr)q^m e^{m u} \frac{z^{2m}}{m!^2}, \label{theta1z}\\
&\theta_2(v,u;q;z):=\sum_{k\geq0} \theta_{2,k}(v,u;q) z^k = z^{-1} \biggl(\sum_{m\geq0} q^m e^{m u+z v} \frac{z^{2m}}{m!^2}-1\biggr), \label{theta2z}
\end{align}
where $\gamma$ is the Euler constant and $\psi$ denotes the digamma function. 
The reader who is familiar with the theory 
of Frobenius manifolds recognizes that the above 
$(\theta_1(v,u;z), \theta_2(v,u;z))$ give a system of the deformed flat coordinates for the Dubrovin connection of the 
Frobenius manifold under consideration. 
It is easy to observe that the $\p/\p t^{1,0}$ flow coincides with~$\p_x$; therefore, we identify $t^{1,0}$ with~$x$. 
We also remind the reader that  
 the integrability of the principal hierarchy~\eqref{phdefp1}
 is guaranteed by the Frobenius manifold structure. Since \eqref{phdefp1} is integrable,  
 one can solve equations in~\eqref{phdefp1} together,  
 yielding solutions of the form $(v=v({\bf t};q), u=u({\bf t};q))$.
Following~\cite{Du96}, define the {\it genus zero two-point correlation functions} $\Omega_{\alpha,i;\beta,j}^{[0]}(v,u)$
by means of the generating series as follows:
\beq\label{two-pt}
\sum_{i,j\geq 0} \Omega_{\alpha,i;\beta,j}^{[0]}(v,u;q) z^i y^j = \frac{1}{z+y} 
\Biggl(\sum_{\rho,\sigma}\frac{\p \theta_\alpha(v,u; q;z)}{\p v^\rho} \eta^{\rho\sigma} \frac{\p \theta_\beta(v,u;q;y)}{\p v^\sigma} 
- \eta_{\alpha\beta} \Biggr),\quad \alpha,\beta=1,2.
\eeq

Solutions~$(v=v({\bf t};q), u=u({\bf t};q))$ 
to the principal hierarchy~\eqref{phdefp1} are characterized by their initial values 
\beq(v({\bf t};q),u({\bf t};q))|_{t^{\alpha,i}=x \delta^{\alpha,1} \delta^{i,0}, \,\alpha=1,2,\, i\geq0}.\eeq
The {\it topological solution} $(v_{\rm top}({\bf t};q),u_{\rm top}({\bf t};q))$ to the principal hierarchy is defined as the 
unique solution subjected to the initial value
\beq
(x,0).
\eeq
It is straightforward to show that this solution satisfies 
\beq
v_{\rm top}({\bf t};q)\big|_{t^{\alpha,i}=0, \, \alpha=1,2, \, i>0} = t^{1,0},\quad 
u_{\rm top}({\bf t};q)\big|_{t^{\alpha,i}=0, \, \alpha=1,2, \, i>0} = t^{2,0}.
\eeq
According to~\cite{Du96}, the topological solution $(v_{\rm top}({\bf t};q),u_{\rm top}({\bf t};q))$
 can alternatively be determined by the following genus zero Euler--Lagrange equation:
\beq\label{ELeq}
 \sum_\alpha \sum_{i} \tilde t^{\alpha,i}  \frac{\p \theta_{\alpha,i}}{\p v^\beta}(v_{\rm top}({\bf t};q), u_{\rm top}({\bf t};q);q) = 0, \quad \beta=1,2,
\eeq
where $\tilde t^{\alpha,i}:=t^{\alpha,i}-\delta^{\alpha,1} \delta^{i,1}$. 

Recall \cite{Du96} (cf.~\cite{DZ-selecta, DZ}) that 
the genus zero free energy of GW invariants of~$\mathbb{P}^1$ has the expression
\beq\label{fp10c}
\F^{\mathbb{P}^1}_0({\bf t};q) = 
\frac12 \sum_{\alpha,\beta} \sum_{i,j} \tilde t^{\alpha,i} \tilde t^{\beta,j}   
\Omega_{\alpha,i;\beta,j}^{[0]}(v_{\rm top}({\bf t};q), u_{\rm top}({\bf t};q);q),
\eeq
which satisfies the following genus zero Virasoro constraints~\cite{DZ-selecta} (cf.~\cite{EX, Gi, LT}):
\beq\label{virap1genus0}
e^{-\e^{-2}\F^{\mathbb{P}^1}_0({\bf t};q)} L_m\Bigl(e^{\e^{-2}\F^{\mathbb{P}^1}_0({\bf t};q)}\Bigr) = {\rm O}(1) \quad (\epsilon\to0), \quad m\geq-1,
\eeq
where $L_m$, $m\geq-1$, are the linear operators given by
\begin{align}
L_{-1} = \, & \sum_{k\geq1}\sum_\alpha \tilde t^{\alpha,k} \frac{\p}{\p t^{\alpha,k-1}} + \frac{t^{1,0}t^{2,0}}{\e^2}, \\
L_m = \, & \e^2 \sum_{k=1}^{m-1} k!(m-k)! \frac{\p^2}{\p t^{2,k-1}\p t^{2,m-k-1}} 
+ \sum_{k\geq1} \frac{(m+k)!}{(k-1)!} \biggl(\tilde t^{1,k}\frac{\p}{\p t^{1,m+k}}+\tilde t^{2,k-1}\frac{\p}{\p t^{2,m+k-1}}\biggr) \nn\\
& + 2 \sum_{k\geq0} A_m(k) \tilde t^{1,k} \frac{\p}{\p t^{2,m+k-1}} + \delta_{m,0} \frac{t^{1,0}t^{1,0}}{\e^2} , \quad m\geq0
\end{align}
with 
\beq
A_m(0)=m!,  \quad A_m(k)=\frac{(m+k)!}{(k-1)!} \sum_{j=k}^{m+k}\frac1j ~ (k>0).
\eeq
These operators $L_m$, $m\geq-1$, satisfy the Virasoro commutation relations: 
\beq
[L_{m_1},L_{m_2}]=(m_1-m_2) L_{m_1+m_2}, \quad m_1,m_2\geq-1.
\eeq

Let us proceed with higher genera. 
According to~\cite{DW, DZ-norm, DZ, EYY, Getzler, Gi},  
the higher genus free energies $\F^{\mathbb{P}^1}_g({\bf t};q)$, $g\geq1$, have the 
$(3g-2)$ property, namely, 
there exist functions 
\beq
F_g=F_g(v, u, v_1, u_1, \dots, v_{3g-2}, u_{3g-2};q), \quad g\geq1,
\eeq 
such that, for $g\ge1$, 
\beq\label{fgfmgequalp1}
\F_g^{\mathbb{P}^1}({\bf t};q) = 
F_g\biggl(v_{\rm top}({\bf t};q), u_{\rm top}({\bf t};q), \frac{\p v_{\rm top}({\bf t};q)}{\p x}, \frac{\p u_{\rm top}({\bf t};q)}{\p x}, \dots, 
\frac{\p^{3g-2} v_{\rm top}({\bf t};q)}{\p x^{3g-2}},
\frac{\p^{3g-2} u_{\rm top}({\bf t};q)}{\p x^{3g-2}}; q\biggr).
\eeq
Moreover, the partition function of the GW invariants of~$\mathbb{P}^1$ satisfies the  
Virasoro constraints~\cite{Gi, OP} (cf.~\cite{DZ-norm, DZ, EHX}):
\beq\label{virap1}
L_m\bigl(Z^{\mathbb{P}^1}({\bf t};\e,q)\bigr) = 0, \quad m\geq0.
\eeq
Here $L_m$, $m\geq-1$, are the above-defined operators. 

Denote 
\beq
\Delta F=\Delta F(v,u,v_1,u_1,v_2,u_2,\dots; \e, q):= \sum_{g\geq 1} \epsilon^{2g-2} F_g(v,u, v_1,u_1,\dots, v_{3g-2}, u_{3g-2};q).
\eeq
Based on formulae~\eqref{fp10c}, \eqref{virap1genus0}, 
the $(3g-2)$ jet-representation~\eqref{fgfmgequalp1}, and the above Virasoro constraints~\eqref{virap1},
 Dubrovin and Zhang derive \cite{DZ-norm, DZ} the loop equation for 
GW invariants of~$\mathbb{P}^1$, given in the following theorem. (The derivation 
is also revisited and slightly simplified in~\cite{Yang}.)

\smallskip

\noindent {\bf Theorem B.} (Dubrovin--Zhang~\cite{DZ-norm, DZ}). 
{\it The function $\Delta F$ satisfies the following loop equation:
\begin{align}
&\sum_{r\geq 0}\biggl({\p \Delta F \over \p v_r}
\Bigl({v-\lambda\over D}\Bigr)_r - 2 {\p \Delta F \over \p u_r} \Bigl({1\over D}\Bigr)_r \biggr) \nn\\
&+\sum_{r\geq 1} \sum_{k=1}^r \binom{r}{k}
\Bigl(\frac1{\sqrt{D}}\Bigr)_{k-1} \biggl( {\p \Delta F \over \p v_r}
\left(v\!-\!\lambda\over \sqrt{D}\right)_{r-k+1}-2 {\p \Delta F \over \p u_r} \left(1\over \sqrt{D}\right)_{r-k+1} \biggr)
\nn\\
&+ \frac{q e^{u}}{\bigl((\lambda-v)^2-4 q e^{u}\bigr)^2} \nn\\
&+ \epsilon^2 \sum_{k,l\geq0} \biggl( \frac14 S(\Delta F,v_k,v_l)
\biggl(\frac{\lambda-v}{\sqrt{D}}\biggr)_{k+1} \biggl(\frac{\lambda-v}{\sqrt{D}}\biggr)_{l+1} \nn\\
&+ S(\Delta F,v_k,u_l) \left(\lambda-v\over
\sqrt{D}\right)_{k+1} \left(1\over \sqrt{D}\right)_{l+1}+ S(\Delta F,u_k,u_l) \,
\left(1\over \sqrt{D}\right)_{k+1} \left(1\over \sqrt{D}\right)_{l+1}\biggr)
\nn\\
&+ \frac{\epsilon^2} 2 \sum_{k\geq0}  \biggl( {\p \Delta F \over \p v_k} \, 
   \biggl(qe^{u} {4 q e^{u}(v-\lambda) u_1 - ((v-\lambda)^2 + 4 q e^{u}) v_1 \over D^3}\biggr)_{k+1}  \nn\\
&\qquad\qquad\qquad +{\p \Delta F \over \p u_k} \, 
\biggl(q e^{u} {4 (v-\lambda) \, v_1 - ((v-\lambda)^2 + 4 q e^{u}) u_1\over D^3}\biggr)_{k+1}\biggr)=0,
\label{loopDZ}
\end{align}
where $D= (v-\lambda)^2 - 4 q e^{u}$,
$S(f,a,b):=
{\p^2 f\over \p a \p b}+
{\p f \over \p a}
{\p f \over \p b}$,
and $f_r$ stands for $\p^r(f)$ with 
\beq
\p:=\sum_{\alpha} \sum_{m} v^\alpha_{m+1} \frac{\p}{\p v^{\alpha}_m}.
\eeq
Moreover, the solution to the above loop equation~\eqref{loopDZ} is unique up to a sequence of additive elements 
$c_1,c_2,\dots\in\CC[[q]]$ that can be determined by
the value 
\beq\label{F1p1}
F_1=\frac1{24} \log (v_1^2-4qe^uu_1^2) -\frac1{24} u
\eeq
and the dilaton equation
\beq\label{Fgp1}
\sum_{\alpha} \sum_{m} m v^\alpha_{m} \frac{\p F_g}{\p v^{\alpha}_m} = (2g-2) F_g~ (g\geq2). 
\eeq}

We note that the identity~\eqref{loopDZ} should be understood to be valid {\it identical} in~$\lambda$ for $\lambda$ large. In particular, the branch of $\sqrt{D}$ is taken so that $\sqrt{D}\sim \lambda$ for $\lambda$ large.

\section{Proof of Theorem~\ref{thm1}} \label{section3}

In this section, we give a proof of Theorem~\ref{thm1} by applying the loop equation~\eqref{loopDZ}. 

\begin{proof}[Proof of Theorem~\ref{thm1}.]
To simplify the notations, let us  
denote $P_i=t^{1,i}$, $Q_i=t^{2,i}$, $i\geq0$, and 
\[{\bf P}=(P_0,P_1,\dots), \quad {\bf Q}=(Q_0,Q_1,\dots).\]
By~\eqref{genfxh} and~\eqref{123} we know that for $g\geq1$, 
\begin{align}
(-1)^g\F^{\mathbb{P}^1}_{g, \,\deg=0}({\bf t};\e)
&= \frac{2^{2g-1}-1}{2^{2g-1}} \frac{|B_{2g}|}{(2g)!}  \bigl\langle 1, V_{2g-2}({\bf t}) \bigr\rangle 
- 2 \mathcal{H}_g(\lambda_{g-1}; {\bf P}). \label{genp1prejet}
\end{align}
Here we recall that ${\bf t}=(t^{\alpha,i})_{\alpha=1,2,\,i\ge0}$. We also note that ${\bf t}^i=P_i 1 + Q_i [{\rm pt}]$, $i\geq0$, 
are the cohomology-valued times for~$\mathbb{P}^1$. Define a power series 
\begin{align}
& U({\bf P}, {\bf Q}) = \sum_{i\ge0} Q_i \frac{\p V({\bf P})}{\p P_i},
\end{align}
and put 
\[ 
U_{m}({\bf P}, {\bf Q})=\frac{\p^m U({\bf P}, {\bf Q})}{\p P_0^m} , \quad m\geq0.
\]

By using Proposition~\ref{prop1} we can write~\eqref{genp1prejet} 
in terms of the two power series $V({\bf P})$ and $U({\bf P}, {\bf Q})$ 
as follows:
\begin{align}
(-1)^g\F^{\mathbb{P}^1}_{g, \,\deg=0}({\bf t};\e)
&= \frac{2^{2g-1}-1}{2^{2g-1}} \frac{|B_{2g}|}{(2g)!}  U_{2g-2}({\bf P},{\bf Q}) 
- 2 W_g(V_1({\bf P}), \dots,V_{2g-1}({\bf P})), \quad g\ge1.
\end{align}
(Observe that the power series $(V({\bf P}), U({\bf P}, {\bf Q}))$ is actually 
 the degree zero part of the topological solution to the principal hierarchy~\eqref{phdefp1} of the $\mathbb{P}^1$-Frobenius manifold.)

Taking the degree zero part (coefficient of $q^0$) in the Dubrovin--Zhang loop equation~\eqref{loopDZ}, we obtain 
\begin{align}
&  \sum_{r=1}^{2g-1} \frac{\p W_g(V_1,\dots,V_{2g-2})}{\p V_r} \Bigl(\frac1{\lambda-V}\Bigr)_r 
= \mathcal{B}_g(\lambda;{\bf V}).
\end{align}

Comparing the coefficients of $(\lambda-V)^{-2}$, \dots, $(\lambda-V)^{-2g}$, we get
\begin{align}
M \cdot \biggl(\frac{\p W_g}{\p V_1}, \dots, \frac{\p W_g}{\p V_{2g-1}}\biggr)^T = \mathcal{C}_g,
\end{align}
where $M$ is an upper-triangular and non-degenerate matrix and $\mathcal{C}_g=(\mathcal{B}_{g,1},\dots,\mathcal{B}_{g,2g-1})^T$ with $\mathcal{B}_{g,j}={\rm Coef}\bigl((\lambda-V)^{-j-1},\mathcal{B}_g(\lambda;{\bf V})\bigr)$ as in Section~\ref{section1}. 
Since $M$ is upper-triangular, it is not difficult to calculate the explicit expression of its inverse, and we find
\beq
(M^{-1})_{ij}
= \sum_{\mu\in \mathcal{P}_{j-i}} \frac{1}{j!}\binom{\ell(\mu)+j-1}{i-1} L(\mu)  \frac{V_{\mu+1}}{V_1^{\ell(\mu)+j}} .
\eeq
The theorem is then proved by noticing from~\eqref{126713}
that $\sum_{m=2}^{2g-1} (m-1) V_m \frac{\p W_g}{\p V_m} = (2g-2) W_g$.
\end{proof}

We note that it is also clear from~\eqref{126713} that \eqref{WghomocB} can be equivalently written as 
\beq
W_g(V_1,\dots,V_{2g-2})= \frac1{2g-2} \sum_{k=1}^{2g-1} k V_k \sum_{j=1}^{2g-1} c_{k,j} \mathcal{B}_{g,j}.
\eeq

\begin{remark}
In our proof, we used the degree zero part of the 
loop equation (or say of Virasoro constraints) together with~\eqref{lambdagformu}. 
We note that 
it is also possible to use the Virasoro constraints themselves as the uniqueness 
in Theorem~B suggests (of course existence of jet-representation is also needed here), and the point is that 
more constraints on the degree zero invariants come from the {\it degree one} part of the Virasoro constraints. 
\end{remark}

\section{Applications}\label{section4}
In this section, we give some applications of Proposition~\ref{prop1} and Theorem~\ref{thm1}.

Let $X$ be a smooth projective curve of genus~$h$. The notations will be the same as in Section~\ref{section1}.
For a partition $\lambda=(\lambda_1,\dots,\lambda_n)$ with length $n$, denote 
\beq
C_{\lambda_1,\dots,\lambda_n}^X(q):=\sum_{d, \, g\ge0 \atop 2g-2=d(2h-2)+|\lambda|}  q^d 
\int_{[X_{g,n,d}]^{\rm virt}} c_1(\mathcal{L}_1)^{\lambda_1} {\rm ev}_1^*([{\rm pt}]) 
\cdots c_1(\mathcal{L}_n)^{\lambda_n} {\rm ev}_n^*([{\rm pt}]).
\eeq
We note that the integrals appearing in the above generating function is often called GW invariants in the {\it stationary sector}. 

Consider the case that $X=E$ is an elliptic curve. 
The following closed formula for the genus~$g$, $g\geq1$, free energy (in all degrees) $\F_g^E(\bt;q)$ for the 
elliptic curve~$E$ is obtained by Buryak~\cite{Bur22}:
\beq\label{buryakformula}
\F_g^E(\bt;q) = \sum_{\lambda\in \mathcal{P}_{2g-2}} \frac{U_\lambda(\bt)}{\prod_{j\geq1} m_j(\lambda)!} C^E_\lambda\bigl(qe^{U(\bt)}\bigr) - \frac{U(\bt)}{24} \delta_{g,1},
\eeq
where 
\beq
U(\bt) := \frac{\p^2\F_0^E(\bt)}{\p t^{1,0}\p t^{1,0}}
\eeq
Note that we have omitted the argument $q$ in $\F_0^E(\bt;q)$ because it actually does not depend on~$q$, and 
that $C^E_\lambda(q)$ vanishes if $|\lambda|$ is odd. 
Explicitly, $U(\bt)$ has the expression:
\beq\label{explicitformulaU}
U(\bt) = \sum_i Q_i \frac{\p V({\bf P})}{\p P_i} + \sum_{i,j} t^{2,i} t^{3,j} \frac{\p^2 V({\bf P})}{\p P_i \p P_j} ,
\eeq
which can be obtained from~\eqref{eq111-710}.
Here $Q_i=t^{4,i}$, $P_i=t^{1,i}$, and $V({\bf P})$ is the power series defined by~\eqref{VT} as before.

By taking the degree zero limit in Buryak's formula~\eqref{buryakformula} we obtain that 
\beq\label{Buraykdegree0}
\F^E_{g, \,\deg=0}({\bf t})  = 
\sum_{\lambda\in \mathcal{P}_{2g-2}} \frac{U_\lambda(\bt)}{\prod_{j\geq1} m_j(\lambda)!} C^E_\lambda(0) - \frac{U(\bt)}{24} \delta_{g,1}, \quad g\geq1.
\eeq
We have the following corollary. 
\begin{corollary}\label{prop2}
For $g\geq1$ and for a partition $\lambda \in \mathcal{P}_{2g-2}$, 
\beq\label{constantterms}
C^E_\lambda(0)=\left\{
\begin{array}{cc} (-1)^g \frac{2^{2g-1}-1}{2^{2g-1}} \frac{|B_{2g}|}{(2g)!} + \frac1{24}\delta_{g,1}, & \lambda=(2g-2), \\ 
0, & {\rm otherwise}. \\ \end{array}\right.
\eeq
\end{corollary}
\begin{proof}
From Proposition~\ref{prop1} and formula~\eqref{explicitformulaU} we know that 
\beq
\F^E_{g, \,\deg=0}({\bf t}) = (-1)^g \frac{2^{2g-1}-1}{2^{2g-1}} \frac{|B_{2g}|}{(2g)!} U_{2g-2}(\bt).
\eeq
The corollary is then proved by comparing this expression with~\eqref{Buraykdegree0}. 
\end{proof}

\begin{example}
Denote 
\begin{align}
&\sigma_{s}(d)=\sum_{a|d} a^s, \quad s\ge0, \\ 
&E_k(q) = \frac{\zeta(1-k)}2 + \sum_{d\geq1} \sigma_{k-1}(d) q^d, \quad k=2,4,6. 
\end{align}
According to~\cite{OP2} (cf.~\cite{Bur22,Dij, KZ}), 
\begin{align}
&C_{(0)}^E(q)=E_2(q)-\frac{\zeta(-1)}{2}, \quad C_{(2)}^E(q) = \frac{E_4(q)}{12}+\frac{E_2(q)^2}{2}, \\
&C_{(1,1)}^E(q)=\frac{7E_6(q)}{180}+\frac23 E_2(q)E_4(q)-\frac83 E_2(q)^3. 
\end{align}
The constant values of these functions agree with~\eqref{constantterms}. 
\end{example}

Actually, by taking the stationary sector part (i.e. by taking all ${\bf t}={\bf 0}$ except for $t^{2h+2,i}$'s) of formula~\eqref{Fprop1} in 
Proposition~\ref{prop1}, we find that formula~\eqref{constantterms} with $E$ replaced by~$X$ is true for an arbitrary smooth curve~$X$. 
This general formula was obtained in~\cite{OP2} from the GW/Hurwitz correspondence (see pages 529--530 therein). 
In particular formula~\eqref{constantterms} is covered by the GW/Hurwitz correspondence, while we used the $\lambda_g$ conjecture. 

It is clear that, by our formulation, on the contrary,
the formula of~$C^E_\lambda(0)$ and Buryak's formula~\eqref{buryakformula} 
lead to a proof of the $\lambda_g$ conjecture.

Let us proceed to consider the $\lambda_{g-1}$ integrals and give a new proof of Theorem~A.

\begin{proof}[Proof of Theorem~A] From the definition~\eqref{defBg710} we know that 
$\mathcal{B}_g(\lambda; {\bf V})$ is a rational function of~$\lambda$, 
which has a $(2g)$th order pole 
at $\lambda=V$. More precisely, 
we deduce from~\eqref{defBg710} that
\begin{align}
& \mathcal{B}_g(\lambda; {\bf V}) \nn\\
& \sim \frac{2^{2g-1}-1}{2^{2g-1}} \frac{|B_{2g}|}{(2g)!}  (2g-1)! \frac{V_1^{2g-2}}{(\lambda-V)^{2g}} 
+ \frac{2^{2g-1}-1}{2^{2g-1}} \frac{|B_{2g}|}{(2g)!} \sum_{k=1}^{2g-2} \binom{2g-2}{k} (k-1)! (2g-1-k)! \frac{V_1^{2g-2}}{(\lambda-V)^{2g}} \nn\\
& \quad - \frac12\sum_{g_1+g_2=g \atop g_1,g_2\geq1} \frac{2^{2g_1-1}-1}{2^{2g_1-1}} \frac{|B_{2g_1}|}{(2g_1)!} \frac{2^{2g_2-1}-1}{2^{2g_2-1}} \frac{|B_{2g_2}|}{(2g_2)!} (2g_1-1)! (2g_2-1)! \frac{V_1^{2g-2}}{(\lambda-V)^{2g}} \nn\\
& = \Biggl( \frac{2^{2g-1}-1}{2^{2g-1}} \frac{|B_{2g}|}{2g} \sum_{k=1}^{2g-1}  \frac1k 
 - \frac1{2^{2g-1}}\sum_{g_1+g_2=g \atop g_1,g_2\geq1} (2^{2g_1-1}-1) \frac{|B_{2g_1}|}{2g_1} 
(2^{2g_2-1}-1) \frac{|B_{2g_2}|}{2g_2} \Biggr) \frac{V_1^{2g-2}}{(\lambda-V)^{2g}}.\nn
\end{align}
Here, $\sim$ means keeping only the most singular terms.
Therefore, using~\eqref{WghomocB}, we find
\begin{align}
& {\rm Coef}(V_{2g-1},W_g(V_1,\dots,V_{2g-1})) \nn\\
& = \Biggl(\frac{2^{2g-1}-1}{2^{2g-1}} \frac{|B_{2g}|}{(2g)!} \sum_{k=1}^{2g-1}  \frac1k 
 - \frac1{2^{2g-1}(2g-1)!}\sum_{g_1+g_2=g \atop g_1,g_2\geq1} (2^{2g_1-1}-1) \frac{|B_{2g_1}|}{2g_1} 
(2^{2g_2-1}-1) \frac{|B_{2g_2}|}{2g_2} \Biggr) \frac{1}{V_1}. \nn
\end{align}
By noticing that $V_k(0,\dots,0, T_{2g-1}, 0, \dots)= T_{2g-1} \delta_{k,2g-1} + \delta_{k,1}$ ($g\ge2$, $1\le k\le 2g-1$),
the theorem is proved.
\end{proof}

\smallskip

\noindent {\bf Acknowledgements.} 
We would like to thank the anonymous referee for very helpful comments which 
help to improve a lot the presentation of the paper. 
The work is partially supported by NSFC No.~12061131014.

\end{document}